\documentclass{amsart}
\usepackage{amssymb}
\usepackage{amsmath}
\usepackage{enumitem}
\newtheorem{lemma}{Lemma}
\newtheorem{theorem}[lemma]{Theorem}
\newtheorem{definition}[lemma]{Definition}
\newcommand{\C}{\mathbf{C}}
\newcommand{\D}{\mathbf{D}}
\newcommand{\BPos}{\mathbf{BPos}}
\newcommand{\OMP}{\mathbf{OMP}}
\newcommand{\EA}{\mathbf{EA}}
\newcommand{\DP}{\mathbf{DP}}
\DeclareMathOperator{\id}{id}
\DeclareMathOperator{\card}{card}

\usepackage{cellspace}
\usepackage{ifpdf}
\ifpdf
\usepackage[all,pdf,2cell]{xy}\UseAllTwocells\SilentMatrices
\else
\usepackage[all,xdvi,2cell]{xy}\UseAllTwocells\SilentMatrices
\fi
\begin{document}
\title[The Kalmbach monad]{Effect algebras are the Eilenberg-Moore category for the Kalmbach monad}
\author{Gejza Jen\v ca}
\address{
Department of Mathematics and Descriptive Geometry\\
Faculty of Civil Engineering\\
Slovak Technical University\\
Radlinsk\' eho 11\\
	Bratislava 813 68\\
	Slovak Republic
}
\email{gejza.jenca@stuba.sk}
\thanks{
This research is supported by grants VEGA G-1/0297/11,G-2/0059/12 of M\v S SR,
Slovakia and by the Slovak Research and Development Agency under the contracts
APVV-0073-10, APVV-0178-11.
}
\begin{abstract}
The Kalmbach monad is the monad that arises from the free-forgetful
adjunction between bounded posets and orthomodular posets. 
We prove that the category of effect algebras is isomorphic to the 
Eilenberg-Moore category for the Kalmbach monad.
\end{abstract}

\maketitle
\section{Introduction}

In \cite{kalmbach1977orthomodular}, Kalmbach proved the following theorem.
\begin{theorem}\label{thm:kalmbach}
Every bounded lattice $L$ can be embedded into an orthomodular lattice $K(L)$.
\end{theorem}
The proof of the theorem is constructive, $K(L)$ is known under the name {\em Kalmbach
extension} or {\em Kalmbach embedding}. In \cite{mayet1995classes}, Mayet and Navara proved
that Theorem \ref{thm:kalmbach} can be generalized: every bounded poset $P$ can be embedded
in an orthomodular poset $K(P)$. 
In fact, as proved by Harding in \cite{harding2004remarks}, this $K$ is then left adjoint to the forgetful
functor from orthomodular posets to bounded posets. This adjunction gives rise to a monad on the category
of bounded posets, which we call the {\em Kalmbach monad}.

For every monad $(T,\eta,\mu)$ on a category $\C$, there is a standard notion {\em Eilenberg-Moore} category $\C^T$
(sometimes called the {\em category of algebras} or the {\em category of modules} for $T$). The category
$\C^T$ comes equipped with a canonical adjunction between $\C$ and $\C^T$ and this adjunction 
gives rise to the original monad $T$ on $\C$.

In the present paper we prove that the Eilenberg-Moore category for the Kalmbach monad is isomorphic to the category
of effect algebras.

\section{Preliminaries}

We assume familiarity with basics of category theory, see
\cite{mac1998categories,awodey2006category} for reference.

\subsection{Bounded posets}
A {\em bounded poset} is a structure $(P,\leq,0,1)$ such that
$\leq$ is a partial order on $P$, $0,1\in P$ are the bottom and top
elements of $(P,\leq)$, respectively.

Let $P_1,P_2$ be bounded posets. A map $\phi:P_1\to P_2$ is a 
{\em morphism of bounded posets} if and only if it satisfies the following
conditions.
\begin{itemize}
\item $\phi(1)=1$ and $\phi(0)=0$.
\item $\phi$ is isotone.
\end{itemize}

The category of bounded posets is denoted by $\BPos$.

\subsection{Effect algebras}

An {\em effect algebra} 
is a partial algebra $(E;\oplus,0,1)$ with a binary 
partial operation $\oplus$ and two nullary operations $0,1$ satisfying
the following conditions.
\begin{enumerate}
\item[(E1)]If $a\oplus b$ is defined, then $b\oplus a$ is defined and
		$a\oplus b=b\oplus a$.
\item[(E2)]If $a\oplus b$ and $(a\oplus b)\oplus c$ are defined, then
		$b\oplus c$ and $a\oplus(b\oplus c)$ are defined and
		$(a\oplus b)\oplus c=a\oplus(b\oplus c)$.
\item[(E3)]For every $a\in E$ there is a unique $a'\in E$ such that
		$a\oplus a'$ exists and $a\oplus a'=1$.
\item[(E4)]If $a\oplus 1$ is defined, then $a=0$.
\end{enumerate}

Effect algebras were introduced by Foulis and Bennett in their paper 
\cite{FouBen:EAaUQL}.

In an effect algebra $E$, we write $a\leq b$ if and only if there is
$c\in E$ such that $a\oplus c=b$.  It is easy to check that for every effect
algebra $E$, $\leq$ is a partial order on $E$.  Moreover, it is possible to introduce
a new partial operation $\ominus$; $b\ominus a$ is defined if and only if $a\leq
b$ and then $a\oplus(b\ominus a)=b$.  It can be proved that, in an effect
algebra, $a\oplus b$ is defined if and only if $a\leq b'$ if and only if $b\leq
a'$. In an effect algebra, we write $a\perp b$ if and only if $a\oplus b$ exists.

Let $E_1$, $E_2$ be effect algebras. A map $\phi:E_1\to E_2$ is called a
{\em morphism of effect algebras} if and only if it satisfies the following conditions.
\begin{itemize}
\item $\phi(1)=1$.
\item If $a\perp b$, then $\phi(a)\perp\phi(b)$ and $\phi(a\oplus b)=\phi(a)\oplus\phi(b)$.
\end{itemize}

The category of effect algebras is denoted by $\EA$. There is an evident forgetful
functor $U:\EA\to\BPos$.

\subsection{D-posets}

In their paper~\cite{KopCho:DP}, Chovanec and K\^ opka introduced
a structure called {\em D-poset}. Their definition
is an abstract algebraic version the {\em D-poset of fuzzy sets},
introduced by K\^ opka in the paper~\cite{Kop:DPFS}.

A D-poset is a system $(P;\leq,\ominus,0,1)$ consisting of a partially
ordered set $P$ bounded by $0$ and $1$ with a partial binary operation
$\ominus$ satisfying the following conditions.
\begin{enumerate}
\item[(D1)] $b\ominus a$ is defined if and only if $a\leq b$.
\item[(D2)] If $a\leq b$, then $b\ominus a\leq b$ and $b\ominus(b\ominus a)=a$.
\item[(D3)] If $a\leq b\leq c$, then $c\ominus b\leq c\ominus a$ and
$(c\ominus a)\ominus (c\ominus b)=b\ominus a$.
\end{enumerate}

Let $D_1,D_2$ be D-posets. A map $\phi:D_1\to D_2$ is called a {\em
morphism of D-posets} if and only if it satisfies the following conditions.
\begin{itemize}
\item $\phi(1)=1$.
\item If $a\leq b$, then $\phi(a)\leq\phi(b)$ and 
$\phi(b\ominus a)=\phi(b)\ominus\phi(a)$.
\end{itemize}

The category of D-posets is denoted by $\DP$.

There is a natural, one-to-one correspondence between D-posets and effect
algebras. Every effect algebra satisfies the conditions (D1)-(D3). When given
a D-poset $(P;\leq,\ominus,0,1)$, one can construct an effect algebra
$(P;\oplus,0,1)$: the domain of $\oplus$ is given by the rule
$a\perp b$ if and only if $a\leq 1\ominus b$ and we then have
$a\oplus b=1\ominus\bigl((1\ominus a)\ominus b\bigr)$. The resulting structure is then
an effect algebra with the same $\ominus$ as the original D-poset.
It is easy to see that this correspondence is, in fact, an isomorphism
of categories $\DP$ and $\EA$.

Another equivalent structure was introduced by Giuntini and
Greuling in~\cite{GiuGre:TaFLfUP}. We refer to~\cite{DvuPul:NTiQS} for more
information on effect algebras and related topics.

The following lemma collects some well-known properties connecting the $\oplus$, $\ominus$ and $'$ operations
in effect algebras (or D-posets). Complete proofs can be found, for example, in Chapter 1 of 
\cite{DvuPul:NTiQS}. We shall use these facts without an explicit reference.
\begin{lemma}~
\begin{enumerate}
\item[(a)] $a\leq b'$ iff $b\leq a'$ iff $a\oplus b$ exists and then
	$$(a\oplus b)'=a'\ominus b=b'\ominus a.$$
\item[(b)] $a\leq b$ iff $a\oplus b'$ exists and then
	$$(b\ominus a)'=a\oplus b'$$
\item[(c)] $a\leq c\ominus b$ iff $b\leq c\ominus a$ iff $a\oplus b\leq c$ and then
	$$c\ominus(a\oplus b)=(c\ominus a)\ominus b=(c\ominus b)\ominus a$$
\item[(d)] $a\leq b\leq c$ iff $a\oplus (c\ominus b)$ exists and then
	$$c\ominus (b\ominus a)=a\oplus (c\ominus b).$$
\end{enumerate}
\end{lemma}

\subsection{Orthomodular posets}

An {\em orthomodular poset} is a structure $(A,\leq,',0,1)$ such that
$(A,\leq,0,1)$ is a bounded poset and $'$ is a unary operation (called {\em orthocomplementation})
satisfying the following conditions.
\begin{itemize}
\item $x\leq y$ implies $y'\leq x'$.
\item $x''=x$.
\item $x\wedge x'=0$.
\item If $x\leq y'$, then $x\vee y$ exists.
\item If $x\leq y$, then $x\vee (x\vee y')'=y$.
\end{itemize}
If $x\leq y'$, we say that $x,y$ are {\em orthogonal}.

Let $A_1,A_2$ be orthomodular posets. A map $\phi:A_1\to A_2$ is called a {\em
morphism of orthomodular posets} if and only if it satisfies the following conditions.
\begin{itemize}
\item $\phi(1)=1$.
\item If $a\leq b'$, then $\phi(a)\leq\phi(b)'$ and 
$\phi(a\vee b)=\phi(a)\vee \phi(b)$.
\end{itemize}
Alternatively, we may define a morphism of orthomodular posets as order preserving, preserving the
orthocomplementation, and preserving joins of orthogonal elements. 

The category of orthomodular posets is denoted by $\OMP$. An {\em orthomodular lattice}
is an orthomodular poset that is a lattice. We remark that the usual category of orthomodular
lattices, with morphisms preserving joins and meets is not a full subcategory of $\OMP$.

If $A$ is an orthomodular poset, then we may introduce a partial $\oplus$ operation on
$A$ by the following rule: $x\oplus y$ exists iff $x\leq y'$ and then $x\oplus y:=x\vee y$.
The resulting structure is then an effect algebra. This gives us the object part of an evident
full and faithful functor $\OMP\to\EA$.

\subsection{Kalmbach construction}

If $C=\{x_1,\dots,x_n\}$ is a finite chain in a poset $P$,
we write $C=[x_1<\dots<x_n]$ to indicate the partial order.

\begin{definition}\cite{kalmbach1977orthomodular}
Let $P$ be a bounded poset, write
$$
K(P)=\{C:C\text{ is a finite chain in $P$ with even number of elements}\}
$$
Define a partial order on $K(P)$ by the following rule:
$$
[x_1<x_2<\dots<x_{2n-1}<x_{2n}]\leq[y_1<y_2<\dots<y_{2k-1}<y_{2k}]
$$
if for every $1\leq i\leq n$ there is $1\leq j\leq k$ such that
$$
y_{2j-1}\leq x_{2i-1}<x_{2i}\leq y_{2j}.
$$
Define a unary operation $C\mapsto C^\perp$ on $K(P)$ to be the symmetric difference
with the set $\{0,1\}$.
\end{definition}

Originally, Kalmbach considered the construction only for lattices. If $P$ is a
bounded lattice, then $K(P)$ is a lattice as well. Moreover,
$(K(P),\wedge,\vee,~',0,1)$ is an orthomodular lattice.  However, as observed by
Harding in \cite{harding2004remarks}, $K$ is not an object part of a functor from the
category of bounded lattices to the category of orthomodular lattices.

On the positive side, for any bounded poset $P$, $K(P)$ is an orthomodular
poset (see \cite{mayet1995classes}) and $K$ can be made to a functor
$K:\BPos\to\OMP$. Indeed, let $f:P\to Q$ be a morphism in $\BPos$ and define
$K(f):K(P)\to K(Q)$ by the following rule. 

For an arrow $f:P\to Q$ in $\BPos$, write 
\begin{multline*}
K(f)([x_1<x_2<\dots<x_{2n-1}<x_{2n}])=\\
\{y\in Q:\card(\{1\leq i\leq n:f(x_i)=y\})\text{ is odd}\}.
\end{multline*}

A more elegant way how to write the same rule is
$$
K(f)([x_1<x_2<\dots<x_{2n-1}<x_{2n}])=\Delta_{i=1}^{2n}\{f(x_i)\},
$$
where $\Delta$ is the symmetric difference of sets.

Then $K$ is a functor. Moreover, as proved by Harding in \cite{harding2004remarks},
$K$ is left-adjoint to the forgetful functor $U:\OMP\to\BPos$.
Since every functor has (up to isomorphism) at most one adjoint, this can
be viewed as an alternative definition of the Kalmbach construction.

The unit of the $K\dashv U$ adjunction is the natural transformation
$\eta:\id_{\BPos}\to UK$, given by the rule 
$$
\eta_P(a)=
\begin{cases}
[0<a]&a>0\\
\emptyset& a=0
\end{cases}
$$
and the counit of the adjunction is the natural transformation
$\epsilon:KU\to\id_{\OMP}$ given by the rule
$$
\epsilon_L([x_1<\dots <x_{2n}])=
(x_1^\perp \wedge x_2)\vee\dots\vee(x_{2n-1}^\perp \wedge x_{2n}).
$$

\section{Kalmbach monad}

Let $\C$ be a category. A {\em monad on $\C$} can be defined as a monoid in the
strict monoidal category of endofunctors of $\C$. Explicitly, a monad on $\C$ is
a triple $(T,\eta,\mu)$, where $T:\C\to\C$ is an endofunctor of $\C$ and 
$\eta,\mu$ are natural transformations $\eta:\id_\C\to T$,
$\mu:T^2\to T$ satisfying the equations $\mu \circ T\mu = \mu \circ \mu T$ and
$\mu \circ T \eta = \mu \circ \eta T = 1_{T}$.

Every adjoint pair of functors $F:\C\rightleftarrows\D:G$, with
$F$ being left adjoint, gives rise to a monad $(FG,\eta,G\epsilon F)$ on $\C$.

Let $(T,\eta,\mu)$ be a monad on a category $\C$. Recall,
that the {\em Eilenberg-Moore category for $(T,\eta,\mu)$} is a category
(denoted by $\C^T$), such that  
objects (called {\em algebras for that monad}) of $\C^T$ are pairs 
$(A,\alpha)$, where $\alpha:T(A)\to A$, such that the diagrams
\begin{equation}
\label{eq:algtriangle}
\xymatrix{
A \ar[r]^{\eta_A} \ar[rd]_{1_A} & T(A) \ar[d]^{\alpha} \\
 & A
}
\end{equation}
\begin{equation}
\label{eq:algsquare}
\xymatrix{
T^2(A) \ar[r]^{T(\alpha)} \ar[d]_{\mu_A} & T(A) \ar[d]^{\alpha}\\
T(A) \ar[r]_{\alpha} & A
}
\end{equation}
commute. A morphism of algebras 
$h:(A_1,\alpha_1)\to (A_2,\alpha_2)$ is a $\C$-morphism such that
the diagram
$$
\xymatrix{
T(A) \ar[r]^{T(h)} \ar[d]_{\alpha_1}& T(B)\ar[d]^{\alpha_2}\\
A\ar[r]^{h}&B
}
$$
commutes.

Consider now the adjunction $K\dashv U$ between the categories, $\BPos$ and $\OMP$
from the preceding section. This adjunction gives rise to a monad on $\BPos$,
which we will denote $(T,\eta,\mu)$. Explicitly, $T=U\circ K$, $\eta$ remains the same
and $\mu=U\epsilon K$ turns out to be given by the following rule
$$
\mu_P\bigl([C_1<C_2<\dots<C_{2n}]\bigr)=\Delta_{i=1}^{2n}C_i,
$$
where $[C_1<C_2<\dots<C_{2n}]$ is a chain of even length of chains of even length, that
means, an element of $T^2(P)$, and $\Delta$ is the symmetric difference of sets.

\begin{theorem}
The category of effect algebras is isomorphic to the Eilenberg-Moore category for the Kalmbach
monad.
\end{theorem}
\begin{proof}
From now on, let $U$ be the forgetful functor $U:\EA\to\BPos$.
Let us define a functor $G:\EA\to\BPos^T$.
For an effect algebra $A$, define $m_A:T(U(A))\to U(A)$ by the rule
$$
m_A\bigl([x_1<x_2<\dots x_{2n-1}<x_{2n}]\bigr)=(x_2\ominus x_1)\oplus\dots\oplus(x_{2n}\ominus x_{2n-1}).
$$
We claim that $G(A)=(U(A),m_A)$ is an algebra for the Kalmbach monad.
We need to prove that the diagrams (\ref{eq:algtriangle}) and (\ref{eq:algsquare}) commute.
Clearly, for every $x\in U(A)$, 
$$
(m_A\circ\eta_{U(A)})(x)=m_A([0<x])=(x\ominus 0)=x,
$$
and we see that the triangle diagram (\ref{eq:algtriangle}) commutes. Consider now the
square diagram (\ref{eq:algsquare}):
the elements of $T^2(U(A))$ are chains of chains of elements of $U(A)$;
let $[C_1<C_2<\dots<C_{2n}]\in T^2(U(A))$. 
Note that $C_i<C_j$ implies that $m_A(C_i)<m_A(C_j)$, so the elements of the
sequence
$(m_A(C_1),\dots,m_A(C_{2n}))$ are pairwise distinct. Therefore,
\begin{multline*}
\bigl(m_A\circ T(m_A)\bigr)\bigl([C_1<C_2<\dots<C_{2n}]\bigr)=\\m_A\bigl(\bigl[m_A(C_1)<m_A(C_2)<\dots <m_A(C_{2n})\bigr]\bigr)=\\
\bigl(m_A(C_2)\ominus m_A(C_1)\bigr)\oplus\dots\oplus \bigl(m_A(C_{2n})\ominus m_A(C_{2n-1})\bigr)
\end{multline*}

Note that, if $C<D$ in $T(U(A))$, then $C\Delta D<D$, $m_A(C)<m_A(D)$ and $m_A(D)\ominus m_A(C)=m_A(C\Delta D)$.
Using these facts,
\begin{multline*}
\bigl(m_A\circ\mu_{U(A)}\bigr)\bigl([C_1<C_2<\dots <C_{2n}]\bigr)=
m_A(C_1\Delta C_2\Delta\dots\Delta C_{2n})=\\
m_A\bigl((C_1\Delta C_2\Delta\dots\Delta C_{2n-1})\Delta C_{2n}\bigr)=\\
m_A(C_{2n})\ominus m_A(C_1\Delta C_2\Delta\dots\Delta C_{2n-1})=\\
m_A(C_{2n})\ominus \bigl(m_A(C_{2n-1})\ominus m_A(C_1\Delta C_2\Delta\dots\Delta C_{2n-2})\bigr)=\\
\bigl(m_A(C_{2n})\ominus m_A(C_{2n-1})\bigr)\oplus m_A(C_1\Delta C_2\Delta\dots\Delta C_{2n-2}).
\end{multline*}
The desired equality now follows by a simple induction.

If $f:A\to B$ is a morphism of effect algebras, we define $G(f)=U(f)$. We need to prove that the
diagram
$$
\xymatrix{
TU(A) \ar[r]^{TU(f)} \ar[d]_{m_A}& TU(A) \ar[d]^{m_B}\\
U(A) \ar[r]_{U(f)} & U(B)
}
$$
commutes. After some simple steps, this reduces to the following equality in $B$:
\begin{equation}
\label{eq:morphalg}
\bigl(f(x_2)\ominus f(x_1)\bigr)\oplus\dots\oplus \bigl(f(x_{2n})\ominus f(x_{2n-1})\bigr)=
m_B\bigl(\Delta_{i=1}^n\{f(x_i)\}\bigr),
\end{equation}
for each $[x_1<\dots<x_{2n}]\in T(U(A))$.

Let us define an auxiliary function $k:T(U(A))\to\mathbb N$:
for $C=[x_1<x_2<\dots<x_{2n}]\in T(U(A))$, $k(C)$ is the number of
equal consecutive pairs in the sequence $\bigl(f(x_1),\dots,f(x_n)\bigr)$, that means,
$k(C)$ is the cardinality of the set $\{i:f(x_i)=f(x_{i+1})\}$. 

To prove the equality (\ref{eq:morphalg}), we use induction with respect to $k(C)$.
If $k(C)=0$, then the equality (\ref{eq:morphalg}) clearly holds.

If $k(C)>0$, then let us pick some $i$ with $f(x_i)=f(x_{i+1})$.
Then $\{f(x_i)\}\Delta \{f(x_{i+1})\}=\emptyset$ and we
may skip them on the right hand side of (\ref{eq:morphalg}).

If $i$ is odd, then $f(x_{i+1})\ominus f(x_i)=0$ and we may delete
that term from the left-hand side of (\ref{eq:morphalg}). If $i$ is even,
then
$$
\bigl(f(x_{i})\ominus f(x_{i-1})\bigr)\oplus\bigl(f(x_{i+2})\ominus f(x_{i+1})\bigr)=
f(x_{i+2})\ominus f(x_{i-1})
$$
and we may simplify the left-hand side of (\ref{eq:morphalg}) accordingly.

So (\ref{eq:morphalg}) is true if and only if it is true for the chain
$C-\{x_i,x_{i+1}\}$. Clearly, $k(C-\{x_i,x_{i+1}\})=k(C)-1$ and we have
completed the induction step.

Let $(A,\alpha)$ be an algebra for the Kalmbach monad. Let us
define a partial operation $\ominus$ on the bounded poset $A$
given by this rule: $b\ominus a$ is defined if and only if
$a\leq b$ and
$$
b\ominus a=
\begin{cases}
0 & a=b\\
\alpha([a<b]) & a<b
\end{cases}
$$
We claim that $E(A,\alpha)=(A,\leq,\ominus,0,1)$ is then a D-poset,
hence an effect algebra.

The axiom (D1) follows by definition.

Before we prove the other two axioms, let us note that for all $a\in A$,
$a\ominus 0=a$. Indeed, if $0<a$ then the triangle diagram
(\ref{eq:algtriangle}) implies that $a=\alpha([0<a])=a\ominus 0$
and for $a=0$ we obtain $a\ominus 0=0\ominus 0=0$ by definition
of $\ominus$.

To prove (D2), let $a,b\in A$ be such that $a\leq b$.

Let us prove that $b\ominus a\leq b$.
If $a<b$, then $[a<b]\leq[0<b]$ in the poset $T(A)$ and
$$
b\ominus a=\alpha\bigl([a<b]\bigr)\leq\alpha\bigl([0<b]\bigr)=b\ominus 0=b.
$$
If $a=b$ then $b\ominus a=0\leq b$.

Let us prove that 
$b\ominus (b\ominus a)=a$
There are three possible cases.
\begin{enumerate}[leftmargin=2em]
\item[(D2.1)] Suppose that $0<a<b$. Then, $[a<b]<[0<b]$ in  $T(A)$ and hence
$\bigl[[a<b]<[0<b]\bigr]\in T^2(A)$.
Suppose that $\alpha\bigl([a<b]\bigr)=\alpha\bigl([0<b]\bigr)$.
From the commutativity of the square (\ref{eq:algsquare})
we obtain
$$
\xymatrix{
\bigl[[a<b]<[0<b]\bigr] \ar@{|->}[r]^-{T(\alpha)} \ar@{|->}[d]_{\mu_A} & \emptyset \ar@{|->}[d]^\alpha \\
[0<a] \ar@{|->}[r]_-\alpha & \alpha\bigl([0<a]\bigr)=0
}
$$
However, $0<a=\alpha\bigl([0<a]\bigr)=0$ is false and we have proved that
$\alpha\bigl([a<b]\bigr)<\alpha\bigl([0<b]\bigr)$. Chasing the element 
$\bigl[[a<b]<[0<b]\bigr]$ around the square 
$$
\xymatrix{
\bigl[[a<b]<[0<b]\bigr] \ar@{|->}[r]^-{T(\alpha)} \ar@{|->}[d]_{\mu_A} & \bigl[\alpha\bigl([a<b]\bigr)<\alpha\bigl([0<b]\bigr)\bigr] \ar@{|->}[d]^\alpha \\
[0<a] \ar@{|->}[r]_-\alpha &\alpha\bigl([0<a]\bigr)=\alpha\Bigl(\bigl[\alpha\bigl([a<b]\bigr)<\alpha\bigl([0<b]\bigr)\bigr]\Bigr)
}
$$
gives us the equality in the bottom right corner,
meaning that $b\ominus(b\ominus a)=a$.
\item[(D2.2)] Suppose that $0=a$. We already know that $b\ominus 0=b$ and
we may compute
$$
b\ominus(b\ominus a)=b\ominus(b\ominus 0)=b\ominus b=0=a.
$$
\item[(D2.3)] Suppose that $a=b$. Then
$$
b\ominus(b\ominus a)=b\ominus 0=b=a.
$$
\end{enumerate}

To prove (D3), let $a,b,c\in A$ be such that $a\leq b\leq c$.

Let us prove that $c\ominus b\leq c\ominus a$.
If $a=b$, there is nothing to prove. If $b=c$, then $b\ominus c=0\leq c\ominus a$.
Assume that $a<b<c$. Then $[b<c]<[a<c]$ and
$$
c\ominus b=\alpha\bigl([b<c]\bigr)\leq\alpha\bigl([a<c]\bigr)=c\ominus a.
$$

Let us prove that $b\ominus a=(c\ominus a)\ominus (c\ominus b)$.
\begin{enumerate}[leftmargin=2em]
\item[(D3.1)]
Suppose that $a<b<c$ and assume that $\alpha([b<c])=\alpha([a<c])$.
The square
$$
\xymatrix{
\bigl[[b<c]<[a<c]\bigr] \ar@{|->}[r]^-{T(\alpha)} \ar@{|->}[d]_{\mu_A} & \emptyset \ar@{|->}[d]^\alpha \\
[a<b] \ar@{|->}[r]_-\alpha &\alpha\bigl([a<b]\bigr)=0
}
$$
gives us $\alpha\bigl([a<b]\bigr)=0$, so $b\ominus a=0$.
However, using only the properties of $\ominus$ we already proved,
$$
b=b\ominus 0=b\ominus(b\ominus a)=a<b,
$$
which is false. Thus, assuming $\alpha\bigl([b<c]\bigr)<\alpha\bigl([a<c]\bigr)$ 
the square
$$
\xymatrix{
\bigl[[b<c]<[a<c]\bigr] \ar@{|->}[r]^-{T(\alpha)} \ar@{|->}[d]_{\mu_A} & \bigl[\alpha\bigl([b<c]\bigr)<\alpha\bigl([a<c]\bigr)\bigr] \ar@{|->}[d]^\alpha \\
[a<b] \ar@{|->}[r]_-\alpha &\alpha\bigl([a<b]\bigr)=\alpha\Bigl(\bigl[\alpha\bigl([b<c]\bigr)<\alpha\bigl([a<c]\bigr)\bigr]\Bigr)
}
$$
gives us the equality
in the bottom right corner
meaning that
$$
b\ominus a=(c\ominus a)\ominus (c\ominus b).
$$
\item[(D3.2)]
Suppose that $b=c$. Then
$$
(c\ominus a)\ominus(c\ominus b)=(c\ominus a)\ominus 0=c\ominus a=b\ominus a.
$$
\item[(D3.3)]
If $a=b$, then there is nothing to prove.
\end{enumerate}

If is now easy to check that an arrow $h:(A,\alpha)\to (B,\beta)$ in $\BPos^T$ is,
at the same time, a morphism of D-posets (and thus, a morphism of effect algebras) $E(A,\alpha)\to E(B,\beta)$.
Indeed, $h(0)=0$ and $h(1)=1$. If $a<b$ and $h(a)<h(b)$, then 
$$
h(b\ominus a)=h\bigl(\alpha([a<b])\bigr)=\beta\bigl(T(h)([a<b])\bigr)=\beta\bigl([h(a)<h(b)]\bigr)=h(b)\ominus h(a).
$$
If $a<b$ and $h(a)=h(b)$ then
$$
h(b\ominus a)=h(\alpha[a<b])=\beta\bigl(T(h)([a<b])\bigr)=\beta(\emptyset)=0=h(b)\ominus h(a).
$$

Therefore $E$ is a functor from $\BPos^T$ to $\EA$.

It remains to prove that $E,G$ are mutually inverse functors. Let 
$A$ be an effect algebra. We claim that $EG(A)=A$. 
The underlying poset of $EG(A)$ and $A$ is the same.
For all $a<b$, 
$$
b\ominus_{EG(A)} a=m_A([a<b])=b\ominus_A a,
$$
hence $EG(A)=A$. It is obvious that for every 
morphism $f:A\to B$ of effect algebras $EG(f)=f$, since
both $E$ and $G$ preserve the underlying poset maps.

Let $(A,\alpha)$ be an algebra for the Kalmbach monad. We claim that
$(A,\alpha)=GE(A,\alpha)$, that means, for all $[x_1<\dots<x_{2k}]\in T(A)$,
\begin{equation}\label{eq:last}
\alpha([x_1<x_2<\dots<x_{2k-1}<x_{2k}])=(x_2\ominus x_1)\oplus\dots\oplus (x_{2k}\ominus x_{2k-1}),
\end{equation}
where the $\oplus,\ominus$ on the right-hand side are taken in $E(A,\alpha)$. 

To prove (\ref{eq:last}) we need an auxiliary claim:
for every $C\in T(A)$ and an upper bound $u$ of $C$ with $u\notin C$,
$\alpha(C\Delta[0<u])=\alpha([\alpha(C)<u])$. This 
is easily seen by chasing the element $[C<[0<u]]\in T^2(A)$ around the square (\ref{eq:algsquare}).

Clearly, equality (\ref{eq:last}) is true for $k=0$. Suppose it is valid for some $k=n\in\mathbb N$.
Then, for $k=n+1$, equality (\ref{eq:last}) is then equivalent to
\begin{multline*}
\alpha([x_1<x_2<\dots<x_{2n-1}<x_{2n}<x_{2n+1}<x_{2n+2}])=\\ \alpha([x_1<x_2<\dots<x_{2n-1}<x_{2n}])\oplus
(x_{2n+2}\ominus x_{2n+1})
\end{multline*}
Put $C=[x_1<x_2<\dots<x_{2n-1}<x_{2n}]$.
Using the definition of $\ominus$ in $E(A,\alpha)$ and applying the auxiliary claim twice we obtain
\begin{align*}
\alpha(C)\oplus (x_{2n+2}\ominus x_{2n+1})&=\\
x_{2n+2}\ominus\bigl(x_{2n+1}\ominus\alpha(C)\bigr)&=\\
x_{2n+2}\ominus\Bigl(\alpha\bigl(\bigl[\alpha(C)<x_{2n+1}\bigr]\bigr)\Bigr)&=\\
x_{2n+2}\ominus\bigl(\alpha\bigl(C\Delta[0<x_{2n+1}]\bigr)\bigr)&=\\
\alpha\Bigl(\bigl[\alpha\bigl(C\Delta[0<x_{2n+1}]\bigr)<x_{2n+2}\bigr]\Bigr)&=\\
\alpha\bigl(\bigl[C\Delta[0<x_{2n+1}]\Delta[0<x_{2n+2}]\bigr]\bigr)&=\\
\alpha\bigl(\bigl[C\Delta[x_{2n+1}<x_{2n+2}]\bigr]\bigr)&=\\
\alpha\bigl([x_1<x_2<\dots<x_{2n-1}<x_{2n}<x_{2n+1}<x_{2n+2}]\bigr)
\end{align*}
It is obvious that for every morphism $f$ in $\BPos^T$, $GE(f)=f$.
\end{proof}

\begin{thebibliography}{10}
\providecommand{\url}[1]{{#1}}
\providecommand{\urlprefix}{URL }
\expandafter\ifx\csname urlstyle\endcsname\relax
  \providecommand{\doi}[1]{DOI~\discretionary{}{}{}#1}\else
  \providecommand{\doi}{DOI~\discretionary{}{}{}\begingroup
  \urlstyle{rm}\Url}\fi

\bibitem{awodey2006category}
Awodey, S.: Category theory.
\newblock No.~49 in Oxford Logic Guides. Oxford University Press (2006)

\bibitem{DvuPul:NTiQS}
Dvure{\v c}enskij, A., Pulmannov{\'a}, S.: New Trends in Quantum Structures.
\newblock Kluwer, Dordrecht and Ister Science, Bratislava (2000)

\bibitem{FouBen:EAaUQL}
Foulis, D., Bennett, M.: Effect algebras and unsharp quantum logics.
\newblock Found. Phys. \textbf{24}, 1325--1346 (1994)

\bibitem{GiuGre:TaFLfUP}
Giuntini, R., Greuling, H.: Toward a formal language for unsharp properties.
\newblock Found. Phys. \textbf{19}, 931--945 (1989)

\bibitem{harding2004remarks}
Harding, J.: Remarks on concrete orthomodular lattices.
\newblock International Journal of Theoretical Physics \textbf{43}(10),
  2149--2168 (2004)

\bibitem{kalmbach1977orthomodular}
Kalmbach, G.: Orthomodular lattices do not satisfy any special lattice
  equation.
\newblock Archiv der Mathematik \textbf{28}(1), 7--8 (1977)

\bibitem{Kop:DPFS}
K{\^o}pka, F.: D-posets of fuzzy sets.
\newblock Tatra Mt. Math. Publ. \textbf{1}, 83--87 (1992)

\bibitem{KopCho:DP}
K{\^o}pka, F., Chovanec, F.: D-posets.
\newblock Math. Slovaca \textbf{44}, 21--34 (1994)

\bibitem{mac1998categories}
Mac~Lane, S.: Categories for the Working Mathematician.
\newblock No.~5 in Graduate Texts in Mathematics. Springer-Verlag (1971)

\bibitem{mayet1995classes}
Mayet, R., Navara, M.: Classes of logics representable as kernels of measures.
\newblock Contributions to General Algebra \textbf{9}, 241--248 (1995)

\end{thebibliography}

\end{document}